\documentclass[12pt]{article}
\usepackage{amssymb,amsthm,amsmath,latexsym}
\usepackage{url}
\newtheorem{thm}{Theorem}
\newtheorem{prop}[thm]{Proposition}

\theoremstyle{remark}

\newcommand{\FF}{\mathbb{F}}

\newcommand{\ww}{\omega}
\newcommand{\vv}{\bar{\omega}}
\def\Aut{\mathop{\rm Aut}}

\begin{document}
\title{New self-dual additive $\FF_4$-codes constructed from circulant
graphs}

\author{
Markus Grassl\thanks{
Max-Planck-Institut f\"ur die Physik des Lichts,
Erlangen, Germany.
email: \protect\url{markus.grassl@mpl.mpg.de}} \and
Masaaki Harada\thanks{
Research Center for Pure and Applied Mathematics,
Graduate School of Information Sciences,
Tohoku University, Sendai 980--8579, Japan.
email: \protect\url{mharada@m.tohoku.ac.jp}.}
}

\maketitle

\begin{abstract}
In order to construct
quantum $[[n,0,d]]$ codes
for 
$(n,d)=(56,15)$, $(57,15)$, $(58,16)$, $(63,16)$, $(67,17)$,
$(70,18)$, $(71,18)$, $(79,19)$, $(83,20)$, 
$(87,20)$, $(89,21)$, $(95,20)$,
we construct self-dual additive $\FF_4$-codes of
length $n$ and minimum weight $d$ from circulant graphs.
The quantum codes with these parameters are constructed
for the first time.
\end{abstract}

\section{Introduction}

Let $\FF_4=\{ 0,1,\ww , \vv \}$ be the finite field with four
elements, where $\vv= \omega^2 = \omega +1$.  An {\em additive}
$\FF_4$-code of length $n$ is an additive subgroup of $\FF_4^n$.
An element of $C$ is called a {\em codeword} of $C$.  An additive
$(n,2^k)$ $\FF_4$-code is an additive $\FF_4$-code of length $n$ with
$2^k$ codewords.  The (Hamming) weight of a codeword $x$ of $C$ is the
number of non-zero components of $x$.  The minimum non-zero weight of
all codewords in $C$ is called the {\em minimum weight} of $C$.

Let $C$ be an additive $\FF_4$-code of length $n$.
The symplectic dual code $C^*$ of $C$ is defined as
$\{x \in \FF_4^n \mid x * y = 0 \text{ for all } y \in C\}$
under the trace inner product:
\[
x * y=\sum_{i=1}^n (x_iy_i^2+x_i^2y_i)
\]
for $x=(x_1,x_2,\ldots,x_n)$, $y=(y_1,y_2,\ldots,y_n) \in \FF_4^n$.
An additive $\FF_4$-code $C$ is called (symplectic) {\em
  self-orthogonal} (resp.\ {\em self-dual}) if $C \subset {C}^*$
(resp.\ $C = {C}^*$).

Calderbank, Rains, Shor and Sloane~\cite{CRSS} gave the following
useful method for constructing quantum codes from self-orthogonal
additive $\FF_4$-codes (see \cite{CRSS} for more details on quantum
codes).  A self-orthogonal additive $(n,2^{n-k})$ $\FF_4$-code $C$
such that there is no element of weight less than $d$ in $C^*\setminus
C$, gives a quantum $[[n,k,d]]$ code, where $k \ne 0$.  In addition, a
self-dual additive $\FF_4$-code of length $n$ and minimum weight $d$
gives a quantum $[[n,0,d]]$ code.  Let $d_{\max}(n,k)$ denote the
maximum integer $d$ such that a quantum $[[n,k,d]]$ code exists.  It
is a fundamental problem to determine the value $d_{\max}(n,k)$ for a
given $(n,k)$.  A table on $d_{\max}(n,k)$ is given in~\cite[Table
  III]{CRSS} for $n \le 30$, and an extended table is available
online~\cite{Grassl}.

In this note, we construct self-dual additive $\FF_4$-codes of
length $n$ and minimum weight $d$ for
\begin{multline}\label{eq:P}
(n,d)=(56,15),(57,15),(58,16),(63,16),(67,17),
\\
(70,18),(71,18),(79,19), (83,20), (87,20),(89,21),(95,20).
\end{multline}
These codes are obtained from adjacency matrices of some circulant
graphs.  The above self-dual additive $\FF_4$-codes allow us to construct
quantum $[[n,0,d]]$ codes for the $(n,d)$ given in \eqref{eq:P}.
These quantum codes improve the previously known lower bounds on
$d_{\max}(n,0)$ for the above $n$.

The data of these new quantum codes has already been included
in~\cite{Grassl}.  All computer calculations in this note were
performed using {\sc Magma}~\cite{Magma}.

\section{Self-dual additive $\FF_4$-codes from circulant graphs}

A {\em graph} $\Gamma$ consists of a finite set $V$ of vertices together with
a set of edges, where an edge is a subset of $V$ of cardinality $2$.
All graphs in this note are simple, that is, 
graphs are undirected without loops and multiple edges.
The {\em adjacency matrix} of a graph $\Gamma$ with 
$V=\{x_1,x_2,\ldots,x_v\}$ is a $v \times v$ matrix $A_\Gamma=(a_{ij})$, where 
$a_{ij}=a_{ji}=1$ if $\{x_i,x_j\}$ is an edge and $a_{ij}=0$ otherwise.
Let $\Gamma$ be a graph and let $A_\Gamma$ be the adjacency matrix of $\Gamma$.
Let $C(\Gamma)$ denote the additive 
$\FF_4$-code generated by the rows of $A_\Gamma+ \ww I$, where
$I$ denotes the identity matrix.
Then  $C(\Gamma)$ is a self-dual additive $\FF_4$-code~\cite{DP06}.

Two additive $\FF_4$-codes $C_1$ and $C_2$ of length $n$
are {\em equivalent}
if there is a map from $S_3^n\rtimes S_n$ sending $C_1$ onto $C_2$,
where the symmetric group $S_n$ acts on the set of the $n$ coordinates
and each copy of the the symmetric group $S_3$ permutes the non-zero
elements $1, \ww, \vv$ of the field in the respective coordinate. 
For any self-dual additive $\FF_4$-code $C$,
it was shown in~\cite[Theorem 6]{DP06} that there is a graph $\Gamma$
such that $C(\Gamma)$ is equivalent to $C$.
Using this characterization, all self-dual additive $\FF_4$-codes
were classified for lengths up to $12$~\cite[Section 5]{DP06}.

An  $n \times n$ matrix is {\em circulant} if
it has the following form:
\begin{equation}\label{eq:circulant}
M=\left( \begin{array}{ccccc}
r_1     &r_2     & \cdots &r_{n-1}&r_{n} \\
r_{n}&r_1     & \cdots &r_{n-2}&r_{n-1} \\
r_{n-1}&r_{n}& \ddots &r_{n-3}&r_{n-2} \\
\vdots  & \vdots &\ddots& \ddots & \vdots \\
r_2    &r_3    & \cdots&r_{n}&r_1
\end{array}
\right).
\end{equation}
Trivially, the matrix $M$ is fully determined by its first row
$(r_1,r_2,\ldots,r_{n})$.  A graph is called {\em circulant} if it has
a circulant adjacency matrix.  For a circulant adjacency matrix of the
form \eqref{eq:circulant}, we have
\begin{alignat}{5}
r_1=0\quad\text{and}\quad r_{i}=r_{n+2-i}\quad\text{for $i=2,\ldots,\lfloor n/2\rfloor$.}\label{eq:symmetry}
\end{alignat}
Circulant graphs and their applications have been widely studied
(see \cite{Monakhova} for a recent survey on this subject).  For
example, it is known that the number of non-isomorphic circulant graphs
is known for orders up to $47$ (see the sequence A049287
in~\cite{OEIS}).  In this note,
we concentrate on self-dual additive $\FF_4$-codes
$C(\Gamma)$ generated by the rows of $A_\Gamma+ \ww I$,
where $A_\Gamma$ are the adjacency matrices of circulant graphs
$\Gamma$.
These codes were studied, for example, in~\cite{LLMW} and~\cite{Var}.

\section{New self-dual additive $\FF_4$-codes and quantum codes from
circulant graphs}

\subsection{Lengths up to 50}

Throughout this section, let $\Gamma$ denote a circulant graph
with adjacency matrix $A_\Gamma$.
Let $C(\Gamma)$ denote the self-dual additive 
$\FF_4$-code generated by the rows of $A_\Gamma+ \ww I$.
Let $d_{\max}^\Gamma(n)$  denote
the maximum integer $d$ such that a self-dual 
additive $\FF_4$-code $C(\Gamma)$ of length $n$ and minimum weight
$d$ exists.
Varbanov~\cite{Var} gave a classification of
self-dual additive $\FF_4$-codes $C(\Gamma)$ 
for lengths $n=13,14,\ldots,29,31,32,33$ and 
determined the values $d_{\max}^\Gamma(n)$ 
for lengths up to $33$.

\begin{table}[thb]
\caption{Self-dual additive $\FF_4$-codes $C(\Gamma_n)$ 
of lengths $n=34,35,\ldots,50$}
\label{Tab:F4-50}
\begin{center}
{\footnotesize
\begin{tabular}{c|c|l|c} 
\noalign{\hrule height0.8pt}
\rule[-1.5ex]{0pt}{4ex}$n$ & $d_{\max}^\Gamma(n)$ &
\multicolumn{1}{c|}{Support of the first row of $A_{\Gamma_n}$}
& $d_{\max}(n,0)$ \\
\hline
34& 10 & 2, 3, 6, 8, 9, 27, 28, 30, 33, 34& 10--12\\
35& \emph{10} & 2, 4, 6, 7, 10, 27, 30, 31, 33, 35& 11--13\\
36& \emph{11} & 2, 3, 4, 5, 7, 9, 13, 14, 24, 25, 29, 31, 33, 34, 35, 36& 12--14\\
37& 11 & 5, 6, 7, 9, 11, 12, 27, 28, 30, 32, 33, 34& 11--14\\
38& 12 & 2, 3, 5, 7, 10, 11, 20, 29, 30, 33, 35, 37, 38& 12--14\\
39& 11 &  2, 4, 5, 6, 7, 10, 11, 30, 31, 34, 35, 36, 37, 39 & 11--14\\
40& 12 &  2, 3, 5, 8, 10, 21, 32, 34, 37, 39, 40 & 12--14\\
41&12&2, 3, 4, 5, 6, 10, 11, 13, 30, 32, 33, 37, 38, 39, 40, 41& 12--15\\
42&12&2, 3, 13, 15, 16, 18, 21, 22, 23, 26, 28, 29, 31, 41, 42& 12--16\\
43& \emph{12}&3, 4, 7, 9, 10, 12, 33, 35, 36, 38, 41, 42 & 13--16\\
44& 14& 4, 5, 8, 10, 13, 17, 18, 21, 23, 25, 28, 29, 33, 36, 38, 41, 42 
& 14--16\\
45& 13 & 2, 4, 5, 9, 10, 12, 14, 15, 17, 18, 20, 27, 29, 30, 32, 33, 35,
& 13--16\\
        &&  37, 38, 42, 43, 45 \\ 
46& 14 & 4, 5, 7, 8, 9, 10, 11, 12, 13, 14, 15, 17, 19, 24, 29, 31, 33, 
& 14--16\\
         &&34, 35, 36, 37, 38, 39, 40, 41, 43, 44 & \\
47&13&4, 8, 11, 13, 14, 15, 34, 35, 36, 38, 41, 45& 13--17\\
48&14&3, 4, 5, 10, 12, 14, 15, 16, 25, 34, 35, 36, 38, 40, 45, 46, 47
& 14--18\\
49&13&4, 5, 7, 8, 9, 10, 13, 14, 37, 38, 41, 42, 43, 44, 46, 47& 13--18\\
50&14&3, 7, 8, 9, 11, 12, 13, 17, 20, 22, 24, 25, 26, 27, 28, 30, 32, & 14--18\\
     &&35, 39, 40, 41, 43, 44, 45, 49 \\
\noalign{\hrule height0.8pt}
\end{tabular}
}
\end{center}
\end{table}

For lengths $n=13,14,\ldots, 50$, by exhaustive search, 
we determined the largest minimum weights $d_{\max}^\Gamma(n)$.
In Table~\ref{Tab:F4-50}, for lengths $n=34,35,\ldots,50$, we list
$d_{\max}^\Gamma(n)$ and an example of a self-dual additive
$\FF_4$-code $C(\Gamma_n)$ having minimum weight $d_{\max}^\Gamma(n)$,
where the support of the first row of the circulant adjacency matrix
$A_{\Gamma_n}$ is given.  Our present state of knowledge about the
upper bound $d_{\max}(n,0)$ on the minimum distance is also listed in
the table.  For most lengths, the self-dual additive $\FF_4$-codes give
quantum $[[n,0,d]]$ codes such that $d=d_{\max}(n,0)$ or $d$ attains
the currently known lower bound on $d_{\max}(n,0)$; 
three exceptions (lengths $35, 36$ and $43$) are
typeset in \emph{italics}.

\begin{table}[thb]
\caption{Weight distribution of $C(\Gamma_{36})$}
\label{Tab:WD36}
\begin{center}
{\footnotesize
\begin{tabular}{c|r||c|r||c|r||c|r} 
\noalign{\hrule height0.8pt}
$i$ & \multicolumn{1}{c||}{$A_i$} &
$i$ & \multicolumn{1}{c||}{$A_i$} &
$i$ & \multicolumn{1}{c||}{$A_i$} &
$i$ & \multicolumn{1}{c}{$A_i$} \\
\hline
0 &               1&17&      16\,145\,280&24&  5\,144\,050\,296&31&  3\,388\,554\,144\\
11&          1\,584&18&      51\,147\,440&25&  7\,408\,053\,504&32&  1\,588\,252\,581\\
12&          9\,936&19&     145\,391\,760&26&  9\,402\,473\,952&33&     577\,571\,712\\
13&         52\,992&20&     370\,815\,624&27& 10\,446\,604\,880&34&     152\,925\,552\\
14&        265\,392&21&     847\,669\,248&28& 10\,073\,332\,800&35&      26\,213\,616\\
15&     1\,168\,032&22&  1\,733\,647\,968&29&  8\,336\,897\,280&36&       2\,179\,688\\
16&     4\,578\,786&23&  3\,165\,414\,336&30&  5\,836\,058\,352&&\\
\noalign{\hrule height0.8pt}
\end{tabular}
}
\end{center}
\end{table}

Note that $d_{\max}^\Gamma(36)=11$.  
For lengths $34,35$ and $36$, 
self-dual additive $\FF_4$-codes $C(\Gamma)$ with minimum weight $10$
were constructed in~\cite{Var}.
For length $36$, we found  a self-dual additive
$\FF_4$-code $C(\Gamma_{36})$ 
of length $36$ and minimum weight $11$ (see Table~\ref{Tab:F4-50}).
The weight distribution of the code $C(\Gamma_{36})$ is listed 
in Table~\ref{Tab:WD36}, where $A_i$ denotes the number of
codewords of weight $i$.

\begin{prop}
The largest minimum weight $d_{\max}^\Gamma(36)$ among all self-dual
additive $\FF_4$-codes $C(\Gamma)$ of length $36$ from circulant
graphs is $11$.
\end{prop}

A self-dual additive $\FF_4$-code is called
Type~II if it is even.  
It is known that a Type~II additive $\FF_4$-code must
have even length.
A self-dual additive $\FF_4$-code, which is not Type~II,
is called Type~I\@.
Although the following proposition is somewhat trivial, 
we give a proof for completeness.

\begin{prop}\label{prop}
Let $C(\Gamma)$ be the self-dual additive 
$\FF_4$-code of even length $n$ 
generated by the rows of $A_\Gamma+ \ww I$,
where $A_\Gamma$ is circulant.
Let $S$ be the support of the first row of $A_\Gamma$.
Then $C(\Gamma)$ is Type~II if and only if $n/2+1 \in S$.
\end{prop}
\begin{proof}
It was shown in~\cite[Theorem~15]{DP06} that the codes $C(\Gamma)$ are
Type~II if and only if all the vertices of $\Gamma$ have odd degree.
For a circulant graph $\Gamma$, the degree of the vertices is constant
and equals the size of the support $S$ of the first row of $A_\Gamma$.
From \eqref{eq:symmetry} it follows that the size of the support $S$
is odd if and only if $r_{n/2+1}=1$, i.e., $n/2+1\in S$.
\end{proof}
Note that \eqref{eq:symmetry} also implies that the size of the
support $S$ of the first row of $A_\gamma$ is always even when $n$ is
odd, i.e., self-dual codes of odd length from circulant graphs cannot
be Type~II.

By Proposition~\ref{prop}, the codes $C(\Gamma_n)$
$(n=38,40,42,44,46,48,50)$ are Type~II\@.
In addition, the other codes in Table~\ref{Tab:F4-50} are Type~I\@.
Let $d_{\max,I}^\Gamma(n)$  denote
the maximum integer $d$ such that a Type~I 
additive $\FF_4$-code $C(\Gamma)$ of length $n$ and minimum weight
$d$ exists.
By exhaustive search, we verified that 
$d_{\max,I}^\Gamma(44)=d_{\max}^\Gamma(44)-2$,
$d_{\max,I}^\Gamma(n)=d_{\max}^\Gamma(n)-1$ 
$(n=38,40,46,48)$ and
$d_{\max,I}^\Gamma(n)=d_{\max}^\Gamma(n)$ $(n=42,50)$.
For $(n,d)=(42,12)$ and $(50,14)$, we list an example of Type~I 
additive $\FF_4$-code $C(\Gamma'_n)$ of length $n$ and minimum weight $d$,
where the support of the first row of
the circulant adjacency matrix $A_{\Gamma'_n}$ is given 
in Table~\ref{Tab:F4-I}.

\begin{table}[thb]
\caption{Type~I additive $\FF_4$-codes $C(\Gamma'_n)$ of lengths $42,50$}
\label{Tab:F4-I}
\begin{center}
{\footnotesize
\begin{tabular}{c|c|l} 
\noalign{\hrule height0.8pt}
\rule[-1.5ex]{0pt}{4ex}$n$ & $d$ &
\multicolumn{1}{c}{Support of the first row of $A_{\Gamma'_n}$} \\
\hline
42& 12 & 2, 3, 5, 6, 8, 11, 12, 13, 31, 32, 33, 36, 38, 39, 41, 42\\
50&14& 5, 6, 7, 9, 10, 11, 12, 20, 32, 40, 41, 42, 43, 45, 46, 47\\
\noalign{\hrule height0.8pt}
\end{tabular}
}
\end{center}
\end{table}

\subsection{Sporadic lengths $n\ge 51$}

For lengths $n \ge 51$, by non-exhaustive search, we tried to find
self-dual additive $\FF_4$-codes $C(\Gamma)$ with large minimum
weight, where $\Gamma$ is a circulant graph.  By this method, we found
new self-dual additive $\FF_4$-codes $C(\Gamma_n)$ of length $n$ and
minimum weight $d$ for
\begin{multline*}
(n,d)=(56,15),(57,15),(58,16),(63,16),(67,17),
\\
(70,18),(71,18),(79,19), (83,20), (87,20),(89,21),(95,20).
\end{multline*}
For each self-dual additive $\FF_4$-code $C(\Gamma_n)$, the support of
the first row of the circulant adjacency matrix $A_{\Gamma_n}$ is
listed in Table~\ref{Tab:F4}.  Additionally, for $n=51,\ldots, 55, 59,
60, 64, 65, 66, 69, 72, \ldots, 78, 81, 82, 84, 88, 94, 100$, we found
self-dual additive $\FF_4$-codes $C(\Gamma_n)$ from circulant graphs matching the known
lower bound on the minimum distance of quantum codes $[[n,0,d]]$.
For the remaining lengths, 
our non-exhaustive computer search failed to discover a 
self-dual additive $\FF_4$-code from a circulant graph
matching the known lower bound.

\begin{table}[thbp]
\caption{New self-dual additive $\FF_4$-codes $C(\Gamma_n)$}
\label{Tab:F4}
\begin{center}
{\small
\begin{tabular}{c|l} 
\noalign{\hrule height0.8pt}
\rule[-1.5ex]{0pt}{4ex}Code & \multicolumn{1}{c}{Support of the first row of $A_{\Gamma_n}$}\\
\hline
$C(\Gamma_{56})$ &
2, 3, 7, 8, 12, 14, 15, 16, 17, 20, 22, 26, 28, 30, 32, 36, 38, 41, 42, 43, \\
&
44, 46, 50, 51, 55, 56
\\ \hline  
$C(\Gamma_{57})$ &
7, 8, 10, 12, 17, 18, 22, 23, 24, 35, 36, 37, 41, 42, 47, 49, 51, 52
\\ \hline  
$C(\Gamma_{58})$ &
2, 3, 7, 10, 13, 14, 15, 17, 21, 25, 27, 29, 30, 31, 33, 35, 39, 43, 45, \\
&
46, 47, 50, 53, 57, 58
\\ \hline  
$C(\Gamma_{63})$ &
2, 5, 6, 9, 13, 14, 15, 16, 17, 19, 46, 48, 49, 50, 51, 52, 56, 59, 60, 63
\\ \hline  
$C(\Gamma_{67})$ &
4, 5, 6, 11, 12, 14, 15, 16, 17, 18, 21, 25, 26, 27, 28, 30, 39, 41, 42, \\
&
43, 44, 48, 51, 52, 53, 54, 55, 57, 58, 63, 64, 65 
\\ \hline  
$C(\Gamma_{70})$ &
2, 6, 7, 8, 11, 12, 13, 14, 15, 17, 19, 20, 21, 22, 23, 24, 28, 29, 30, 32,\\
&
33, 35, 36, 37, 39, 40, 42, 43, 44, 48, 49, 50, 51, 52, 53, 55, 57, 58, 59,\\
&
60, 61, 64, 65, 66, 70
\\ \hline  
$C(\Gamma_{71})$ &
2, 3, 5, 11, 12, 15, 17, 20, 23, 26, 27, 28, 31, 34, 35, 38, 39, 42, 45, 46,\\
&
47, 50, 53, 56, 58, 61, 62, 68, 70, 71
\\ \hline  
$C(\Gamma_{79})$ &
2, 4, 7, 10, 13, 15, 18, 19, 20, 21, 23, 24, 25, 29, 30, 31, 32, 35, 36, 37, \\
&
39, 42, 44, 45, 46, 49, 50, 51, 52, 56, 57, 58, 60, 61, 62, 63, 66, 68, 71, \\
&
74, 77, 79
\\ \hline  
$C(\Gamma_{83})$ &
3, 4, 5, 7, 9, 11, 14, 19, 20, 21, 22, 23, 24, 27, 28, 30, 31, 32, 33, 34,\\
& 36, 38, 41, 44, 47, 49, 51, 52, 53, 54, 55, 57, 58, 61, 62, 63, 64, 65, 66,\\
& 71, 74, 76, 78, 80, 81, 82
\\ \hline  
$C(\Gamma_{87})$ &
7, 11, 12, 13, 14, 15, 20, 23, 24, 25, 27, 28, 29, 30, 31, 34, 35, 37, 40,\\ 
&
41, 42, 47, 48, 49, 52, 54, 55, 58, 59, 60, 61, 62, 64, 65, 66, 69, 74, 75,\\
&
76, 77, 78, 82
\\ \hline  
$C(\Gamma_{89})$ &
3, 4, 7, 10, 14, 15, 18, 19, 21, 23, 25, 26, 30, 32, 34, 35, 37, 39, 40,\\
& 45, 46, 51, 52, 54, 56, 57, 59, 61, 65, 66, 68, 70, 72, 73, 76, 77, 81,\\
& 84, 87, 88
\\ \hline  
$C(\Gamma_{95})$ &
4, 5, 6, 11, 12, 14, 15, 18, 19, 26, 27, 28, 30, 31, 32, 33, 34, 35, 36,\\
& 38, 40, 42, 43, 45, 47, 50, 52, 54, 55, 57, 59, 61, 62, 63, 64, 65, 66,\\
& 67, 69, 70, 71, 78, 79, 82, 83, 85, 86, 91, 92, 93\\
\noalign{\hrule height0.8pt}
\end{tabular}
}
\end{center}
\end{table}

\begin{table}[thb]
\caption{Number $A_i$ of codewords of weight $i$ $(i=15,16,\ldots,19)$}
\label{Tab:WD}
\begin{center}
{\small
\begin{tabular}{c|c|r|r|r|r|r} 
\noalign{\hrule height0.8pt}
Code & $d$
& \multicolumn{1}{c|}{$A_{15}$}
& \multicolumn{1}{c|}{$A_{16}$}
& \multicolumn{1}{c|}{$A_{17}$}
& \multicolumn{1}{c|}{$A_{18}$}
& \multicolumn{1}{c}{$A_{19}$} \\
\hline
$C(\Gamma_{56})$ & 15& 4\,032 & 25\,508 & 173\,264 & 1\,124\,648 & 6\,839\,224 \\
$C(\Gamma_{57})$ & 15& 1\,938 &18\,126 &120\,783 & 838\,451 & 5\,093\,409 \\
$C(\Gamma_{58})$ & 16& & 24\,882 & 0 & 1\,205\,240 & 0 \\
$C(\Gamma_{63})$ & 16& & 2\,142 &12\,726 &113\,568&757\,575\\
$C(\Gamma_{67})$ & 17& & & 2\,278 &23\,785 &193\,429 \\
$C(\Gamma_{70})$ & 18& & &  &15\,260  & 0 \\
$C(\Gamma_{71})$ & 18& & &  & 6\,745  &43\,949 \\
$C(\Gamma_{79})$ & 19& & &  & & 1\,343  \\
\noalign{\hrule height0.8pt}
\end{tabular}
}
\end{center}
\end{table}

For the codes $C(\Gamma_{n})$ 
$(n=56,57,58,63,67,70,71,79)$,
we give in Table~\ref{Tab:WD} part of the weight distribution.
Due to the computational complexity, 
we calculated the number $A_i$ of codewords of weight $i$
for only $i=15,16,\ldots,19$.
As some basic properties of the graphs $\Gamma_n$, we give in
Table~\ref{Tab:G} the valency $k(\Gamma_n)$, the diameter
$d(\Gamma_n)$, the girth $g(\Gamma_n)$, the size $\ww(\Gamma_n)$ of
the maximum clique and the order $|\Aut(\Gamma_n)|$ of the
automorphism group.  With the exception of $n=53$, the automorphism
group is the dihedral group on $n$ points of order $2n$.  Note,
however, that the notion of equivalence for graphs and codes are
different, i.\,e., the graph invariants are not preserved with respect
to code equivalence \cite{BCGJWZ}.  
By Proposition~\ref{prop}, 
the codes $C(\Gamma_{58})$ and $C(\Gamma_{70})$ are Type~II\@.


\begin{table}[thbp]
\caption{Properties of the graphs $\Gamma_n$}
\label{Tab:G}
\medskip

\centerline{\small
\begin{tabular}{c|c|c|c|c|c|c} 
\noalign{\hrule height0.8pt}
\rule[-1.5ex]{0pt}{4ex}Graph & 
$d_{\text{min}}(C(\Gamma_n))$ &$k(\Gamma_n)$& $d(\Gamma_n)$&$g(\Gamma_n)$&$\ww(\Gamma_n)$&$|\Aut(\Gamma_n)|$\\
\hline
$\Gamma_{51}$ &  14 & 24 & 2 & 3 & 6 & 102 \\
$\Gamma_{52}$ &  14 & 16 & 3 & 3 & 4 & 104 \\
$\Gamma_{53}$ &  15 & 26 & 2 & 3 & 5 & 1378 \\
$\Gamma_{54}$ &  16 & 29 & 2 & 3 & 8 & 108 \\
$\Gamma_{55}$ &  14 & 14 & 3 & 3 & 4 & 110 \\
$\Gamma_{56}$ &  15 & 26 & 2 & 3 & 19 & 112 \\
$\Gamma_{57}$ &  15 & 18 & 2 & 3 & 5 & 114 \\
$\Gamma_{58}$ &  16 & 25 & 2 & 3 & 7 & 116 \\
$\Gamma_{59}$ &  15 & 30 & 2 & 3 & 8 & 118 \\
$\Gamma_{60}$ &  16 & 31 & 2 & 3 & 6 & 120 \\
$\Gamma_{63}$ &  16 & 20 & 2 & 3 & 5 & 126 \\
$\Gamma_{64}$ &  16 & 43 & 2 & 3 & 12 & 128 \\
$\Gamma_{65}$ &  16 & 28 & 2 & 3 & 6 & 130 \\
$\Gamma_{66}$ &  16 & 33 & 2 & 3 & 6 & 132 \\
$\Gamma_{67}$ &  17 & 32 & 2 & 3 & 6 & 134 \\
$\Gamma_{69}$ &  17 & 38 & 2 & 3 & 7 & 138 \\
$\Gamma_{70}$ &  18 & 45 & 2 & 3 & 10 & 140 \\
$\Gamma_{71}$ &  18 & 30 & 2 & 3 & 6 & 142 \\
$\Gamma_{72}$ &  18 & 27 & 2 & 3 & 6 & 144 \\
$\Gamma_{73}$ &  18 & 40 & 2 & 3 & 8 & 146 \\
$\Gamma_{74}$ &  18 & 32 & 2 & 3 & 6 & 148 \\
$\Gamma_{75}$ &  18 & 34 & 2 & 3 & 6 & 150 \\
$\Gamma_{76}$ &  18 & 37 & 2 & 3 & 8 & 152 \\
$\Gamma_{77}$ &  18 & 48 & 2 & 3 & 10 & 154 \\
$\Gamma_{78}$ &  18 & 35 & 2 & 3 & 7 & 156 \\
$\Gamma_{79}$ &  19 & 42 & 2 & 3 & 8 & 158 \\
$\Gamma_{81}$ &  19 & 40 & 2 & 3 & 7 & 162 \\
$\Gamma_{82}$ &  20 & 43 & 2 & 3 & 7 & 164 \\
$\Gamma_{83}$ &  20 & 46 & 2 & 3 & 9 & 166 \\
$\Gamma_{84}$ &  20 & 25 & 2 & 3 & 6 & 168 \\
$\Gamma_{87}$ &  20 & 42 & 2 & 3 & 7 & 174 \\
$\Gamma_{88}$ &  20 & 37 & 2 & 3 & 6 & 176 \\
$\Gamma_{89}$ &  21 & 40 & 2 & 3 & 6 & 178 \\
$\Gamma_{94}$ &  20 & 44 & 2 & 3 & 10 & 188 \\
$\Gamma_{95}$ &  20 & 50 & 2 & 3 & 7 & 190 \\
$\Gamma_{100}$ & 20 & 48 & 2 & 3 & 7 & 200 \\
\noalign{\hrule height0.8pt}
\end{tabular}
}
\end{table}

Finally, by the method in~\cite{CRSS}, the existence of 
our self-dual additive $\FF_4$-codes $C(\Gamma_n)$
yields the following:

\begin{thm}
There are a quantum $[[n,0,d]]$ codes for
\begin{multline*}
(n,d)=(56,15),(57,15),(58,16),(63,16),(67,17),
\\
(70,18),(71,18),(79,19), (83,20), (87,20),(89,21),(95,20).
\end{multline*}
\end{thm}

The above quantum $[[n,0,d]]$ codes
improve the previously known lower bounds on 
$d_{\max}(n,0)$ ($n=56,57,58,63,67,70,71,79,87,89$).
More precisely, we give
our present state of knowledge about 
$d_{\max}(n,0)$~\cite{Grassl}:
\begin{align*}
& 15 \le d_{\max}(56,0) \le 20,\quad
  15 \le d_{\max}(57,0) \le 20,\\
& 16 \le d_{\max}(58,0) \le 20,\quad
  16 \le d_{\max}(63,0) \le 22,\\
& 17 \le d_{\max}(67,0) \le 24,\quad
  18 \le d_{\max}(70,0) \le 24,\\
& 18 \le d_{\max}(71,0) \le 25,\quad
  19 \le d_{\max}(79,0) \le 28,\\
& 20 \le d_{\max}(83,0) \le 29,\quad
  20 \le d_{\max}(87,0) \le 30,\\
& 21 \le d_{\max}(89,0) \le 31,\quad
  20 \le d_{\max}(95,0) \le 33.\\
\end{align*}

\bigskip
\noindent
{\bf Acknowledgment.}
The authors would like to thank the anonymous referees
for helpful comments.
This work is supported by JSPS KAKENHI Grant Number 15H03633.




\begin{thebibliography}{9}
\bibitem{Magma}W. Bosma, J. Cannon and C. Playoust, 
The Magma algebra system I: The user language, 
{\sl J. Symbolic Comput.}
{\bf 24} (1997), 235--265.

\bibitem{BCGJWZ} S. Beigi, J. Chen, M. Grassl, Z. Ji, Q. Wang and B. Zeng,
Symmetries of codeword stabilized quantum codes,
{\sl  Proceedings 8th Conference on the Theory of Quantum Computation, Communication and Crypto\-graphy (TQC 2013)},
Guelph, Canada, May 2013, pp. 192--206,
preprint arXiv:1303.7020 [quant-ph].

\bibitem{CRSS}A. R. Calderbank, E. M. Rains, P. W. Shor and N. J. A. Sloane,
Quantum error correction via codes over ${\rm GF}(4)$,
{\sl IEEE Trans.\ Inform.\ Theory}
{\bf 44}  (1998),  1369--1387.

\bibitem{DP06}L. E. Danielsen and M. G. Parker,
On the classification of all self-dual additive codes
over ${\rm GF}(4)$ of length up to $12$,
{\sl J. Combin.\ Theory Ser.~A}
{\bf 113} (2006), 1351--1367.

\bibitem{Grassl} M. Grassl,
Bounds on the minimum distance of linear codes and quantum codes,
Online available at \url{http://www.codetables.de},
Accessed on 2015-09-15.


\bibitem{LLMW}R. Li, X. Li, Y. Mao and M. Wei,
Additive codes over $GF(4)$ from circulant graphs,
preprint, arXiv:1403.7933.


\bibitem{Monakhova} E. A. Monakhova,
A survey on undirected circulant graphs,
{\sl Discrete Math.\ Algorithms Appl.}
{\bf 4} (2012), 1250002 (30 pages). 

\bibitem{OEIS}
The OEIS Foundation Inc., 
The on-line encyclopedia of integer sequences, 
Online available at \url{https://oeis.org},
Accessed on 2015-06-25.

\bibitem{Var}Z. Varbanov, 
Additive circulant graph codes over GF($4$),
{\sl Math.\ Maced.}
{\bf 6}  (2008), 73--79. 

\end{thebibliography}
\end{document}